\documentclass[12pt,a4paper,oneside,reqno]{amsart}

\usepackage{pdflscape}
\usepackage{subcaption}
\usepackage{graphicx}
\usepackage{amsmath} % Lots of maths functionality
\usepackage{amssymb} % Maths symbols
\usepackage{mathrsfs} % Maths scripted font.
\usepackage{amsthm} % Maths environments: \begin{proof}, etc.
\usepackage[english]{babel} % Language and hyphenation.
\usepackage[T1]{fontenc} % For font encoding, to allow accents, copy & paste, inequality signs, etc. to all work nicely.
\usepackage[utf8]{inputenc} % To be loaded after fontenc, also for encoding.
\usepackage{mathtools} % Uses amsmath, fixes quirks and adds functionality.
\usepackage[pdftex,  colorlinks=false]{hyperref} % Makes references and citations into links
\usepackage{setspace} % Allows \onehalfspacing etc. for changing gaps between lines
\usepackage{tikz-cd} % Commutative diagrams
\usepackage[capitalize]{cleveref} % use \Cref{} for instead of X~\ref{} 

\usepackage{pgf,tikz,pgfplots}
\usetikzlibrary{arrows,quotes,calc}
\tikzstyle{blackNode}=[fill=black, draw=black, shape=circle]

\numberwithin{equation}{section}

\makeatletter
\@namedef{subjclassname@1991}{Mathematical subject classification 1991}
\@namedef{subjclassname@2000}{Mathematical subject classification 2000}
\@namedef{subjclassname@2010}{Mathematical subject classification 2010}
\@namedef{subjclassname@2020}{Mathematical subject classification 2020}
\makeatother

\newtheorem{lemma}{Lemma}

\newtheorem*{question*}{Question}

\newtheorem{thmx}{Theorem}

\newtheorem*{meta-question*}{Meta-question}

\theoremstyle{definition}
\newtheorem{definition}{Definition}

\newtheorem*{examples*}{Examples}

\newtheoremstyle{TheoremNum}
{\topsep}{\topsep} %%% space between body and thm
{\itshape} %%% Thm body font
{-0.25cm} %%% Indent amount (empty = no indent)
{\bfseries} %%% Thm head font
{.} %%% Punctuation after thm head
{ }  %%% Space after thm head
{\thmname{#1}\thmnote{ \bfseries #3}}%%% Thm head spec
\theoremstyle{TheoremNum}

\DeclareMathOperator{\Dic}{\mathrm{Dic}}
\DeclareMathOperator{\SD}{\mathrm{SD}}

\newcommand{\GL}{\mathrm{GL}}

\newcommand{\SL}{\mathrm{SL}}

\usepackage[foot]{amsaddr}
\title[Solvable and non-solvable groups of the same order type]{Solvable and non-solvable finite groups\\ of the same order type}
\author{Pawe\l\ Piwek}
\date{February 2024}
\address{Mathematical Institute, Andrew Wiles Building, Observatory Quarter, University of Oxford, Oxford OX2 6GG, UK}
\email{pawel.piwek@maths.ox.ac.uk}

\begin{document}
	
\newpage
	
\maketitle

\begin{abstract}
	We construct two finite groups
	of size $2^{365}\cdot 3^{105}\cdot 7^{104}$:
	a solvable group $G$ and a non-solvable group $H$,
	such that for every integer $n$
	the groups have the same number of elements of order $n$.
	This answers a question posed in 1987 by John G. Thompson.
\end{abstract}
	
\section{Introduction}\label{sec:intro}

The following meta-question has been of considerable interest
in the field of finite group theory.

\begin{meta-question*}
	Which properties of a group $G$ can we discern from
	knowing only the orders of its elements?
\end{meta-question*}

While some addressed this question by looking at
just the \emph{set} of orders of elements of $G$ --
see \cite{vasil2009characterization},
we are interested in also using
the \emph{multiplicity} of these orders,
which prompts the following definition.

\begin{definition}
	For a group $G$ we define its \emph{order type} to be the function
	\begin{equation*}
		o_G: n \mapsto |\{g \in G\ |\ g\text{ has order }n\}|.
	\end{equation*}
\end{definition}

It isn't difficult to give examples
of an abelian group $G$ and a non-abelian group $H$
such that $o_G = o_H$.
Indeed, take $G = C_3 \times C_3 \times C_3$
and $H$ to be the Heisenberg group
$(C_3 \times C_3) \rtimes C_3$,
which is non-abelian of exponent $3$.

On the other hand, the order type allows us to count
the number of elements of order $p^\alpha$ for $p$ a prime,
so it can also tell us if a $p$-Sylow subgroup is normal.
This implies that the order type can distinguish nilpotent groups
from non-nilpotent ones.
A similar argument can be given for super-solvable groups.

A problem posed in 1987 by John G. Thompson
in a private communication with Wujie Shi
(see \cite{shen2023thompson, shi2024quantitative};
also Question~2.13 in \cite{guralnick1993same},
and Problem~12.37
in \cite{khukhro2014unsolved})
is: which is the case for solvability?

\begin{question*}
	Let $G$ be a finite solvable group
	and $H$ be any finite group
	such that $o_G = o_H$.
	Is $H$ necessarily solvable?
\end{question*}

It was shown that the answer is positive
if we assume that the {prime graph} of $G$ is disconnected
-- see \cite{shen2023thompson},
or in the case where there are at most three different orders of elements
-- see \cite{shen2012groups}.
Some works also showed results about solvability of groups
(and super-solvability, nilpotence and cyclicity too)
based on the average element order
-- see \cite{herzog2022another,lazorec2023average}.

In this paper, we give a negative answer to Thompson's question.
No claim is made
that the presented examples
are of minimal size.

\begin{thmx}\label{mainthm}
	Let $G_i$ and $H_i$ be the collections of finite groups
	described in \cref{tab:Gis,tab:His},
	and let $m_i$ and $n_i$ be the associated natural numbers from the tables.
	Let $G$ and $H$ be the direct products
	\begin{equation*}
		G = \prod G_i^{m_i}, \quad
		H = \prod H_i^{n_i}.
	\end{equation*}
	Then $o_G = o_H$, $G$ is solvable, and $H$ is not solvable.
\end{thmx}

\begin{table}[h]
	\centering
	\small{
		\begin{subtable}{.45\linewidth}
			\centering
			\begin{tabular}{llll}
				$i$ & $G_i$ & Id &  $m_i$ \\
				\hline
				1 & $C_4$ 		&       (4, 1) &             9 \\
				2 & $D_3$ 		&       (6, 1) &             6 \\
				3 & $C_7$ 		&       (7, 1) &             1 \\
				4 & $D_4$ 		&       (8, 3) &             9 \\
				5 & $D_7$ 		&      (14, 1) &            18 \\
				6 & $\SL(2,3)$	&      (24, 3) &            21 \\
				7 & $C_{24}\rtimes C_2$	&      (48, 6) &             3 \\
				8 & $C_7 \rtimes D_4$	&      (56, 7) &             3 \\
				9 & $C_7\rtimes C_{12}$	&      (84, 1) &             6 \\
				10 & $\Dic_{21}$		&      (84, 5) &             6 \\
				11 & $C_7 \rtimes A_4$	&     (84, 11) &            21 \\
				12 & $C_7\rtimes D_7$	&      (98, 4) &             2 \\
				13 & $C_4 \rtimes F_7$	&     (168, 9) &            21 \\
				14 & $C_{21}\rtimes D_4$&    (168, 15) &             9 \\
				15 & $C_7\rtimes D_{12}$&    (168, 17) &             6 \\
				16 & $F_8\rtimes C_3$	&    (168, 43) &             3 \\
				17 & $D_8\rtimes D_7$	&   (224, 106) &             3 \\
				18 & $C_7\rtimes D_{24}$&    (336, 31) &             3 \\
			\end{tabular}
			\caption{Groups $G_i$ and numbers $m_i$.}
			\label{tab:Gis}
		\end{subtable}
		\hfill
		\begin{subtable}{.45\linewidth}
			\centering
			\begin{tabular}{llll}
				$i$ & $H_i$ & Id &  $n_i$ \\\hline
				1 & $C_2$		&       (2, 1) &            21 \\
				2 & $C_3$		&       (3, 1) &             3 \\
				3 & $\Dic_3$	&      (12, 1) &             6 \\
				4 & $A_4$		&      (12, 3) &            21 \\
				5 & $\SD_{16}$	&      (16, 8) &             3 \\
				6 & $C_7\rtimes C_3$	&      (21, 1) &             4 \\
				7 & $D_{12}$	&      (24, 6) &             6 \\
				8 & $C_3\rtimes D_4$	&      (24, 8) &             6 \\
				9 & $\Dic_7$	&      (28, 1) &            15 \\
				10 & $F_7$		&      (42, 1) &            18 \\
				11 & $D_{21}$	&      (42, 5) &             6 \\
				12 & $D_{24}$	&      (48, 7) &             3 \\
				13 & $D_{28}$	&      (56, 5) &            27 \\
				14 & $\Dic_7\rtimes C_6$	&    (168, 11) &             3 \\
				15 & $C_{14}.A_4$		&    (168, 23) &            21 \\
				16 & $\GL(3, 2)$		&    (168, 42) &             3 \\
				17 & $C_7\rtimes F_7$	&    (294, 10) &             2 \\
				18 & $D_{12}.D_7$		&    (336, 36) &             3 \\
			\end{tabular}
			\caption{Groups $H_i$ and numbers $n_i$.}
			\label{tab:His}
		\end{subtable}
		\label{tab:GisHis}
		\caption{The groups $G_i$ and $H_i$ involved in \cref{mainthm}
			and their multiplicities $m_i$ and $n_i$.
			The column labelled `Id' contains
			the \texttt{Small Groups} isomorphism type identifier
			(see \cite{besche2022smallgrp}),
			while columns $G_i$ and $H_i$
			give some idea of the group structure,
			but they don't necessarily identify the groups uniquely
			in the case of extensions.}
	}
\end{table}

\subsection*{Acknowledgements}

The author is grateful to his friend and colleague Adam Klukowski
for a fruitful conversation about the problem,
and to his PhD supervisor Martin Bridson for the general guidance.

\section{Proof of \cref{mainthm}}\label{sec:proof}

Before presenting the proof, let's first rephrase the question slightly.

\begin{definition}
	For a group $G$ we define its \emph{exponent type} to be the function
	\begin{equation*}
		e_G: n \mapsto |\{g \in G\ |\ g^n = 1\}|.
	\end{equation*}
\end{definition}

\begin{lemma}
	For groups $G$ and $H$ we have: $o_G = o_H \iff e_G = e_H$.
\end{lemma}

\begin{proof}
	$g^n = 1$ if and only if $\text{ord}(g)$ divides $n$.
	Thus $e_G(n) = \sum_{d|n} o_G(d)$,
	and $o_G(n) = \sum_{d|n} e_G(n/d) \cdot \mu(d)$,
	where $\mu$ is the M\"{o}bius function.
\end{proof}

\begin{lemma}\label{product}
	Let $G$ and $H$ be any groups. Then $e_{G \times H} = e_G \cdot e_H$.
\end{lemma}

\begin{proof}
	For $(g, h) \in G\times H$ we have $(g, h)^n = (g^n, h^n)$,
	so $(g, h)^n$ is trivial if and only if $g^n$ and $h^n$ are trivial.
\end{proof}

\begin{proof}[Proof of \cref{mainthm}]
	With the help of \texttt{MAGMA} computational package \cite{MR1484478}
	we can check that---with the notable exception of $H_{16}$---%
	all of the groups $G_i$ and $H_i$ are solvable.
	This means that $G$ is solvable as a product of solvable groups,
	and $H$ isn't as it contains a non-solvable subgroup,
	namely $H_{16} \cong \GL(3, 2)$.
	
	To prove $e_G = e_H$ we need to compute
	the exponent types $e_{G_i}$ and $e_{H_i}$,
	and check that $\prod e_{G_i}^{m_i} = \prod e_{H_i}^{n_i}$.
	
	The exponent types of groups $G_i$ and $H_i$
	were computed with the help of \texttt{MAGMA} too
	(see the ancillary file \texttt{exponents\_computation.m})
	and are listed in \cref{tab:expontent_Gis,tab:expontent_His}.
	Only the columns labelled with divisors of $168$ were included
	as $e_G(n) = e_G\left(\gcd(n, E_G)\right)$
	for $E_G$ being the exponent of the group $G$.
	Since the products $\prod e_{G_i}(n)^{m_i}$
	and $\prod e_{H_i}(n)^{n_i}$
	become unmanageably large,
	we factorise them into their prime factors
	to compute the exponent types of $G$ and $H$
	listed in \cref{tab:Gexps,tab:Hexps}.
	This computation is done in the ancillary file
	\texttt{exponents\_verification.ipynb}.
\end{proof}

\section{How these groups were found}

Finding groups $G$ and $H$ presented in \cref{mainthm} manually
would be very difficult.
Instead, we employed a computer-based search method outlined below.

\textbf{Step 1.}
	We used \texttt{MAGMA}
	to access the \texttt{Small Groups} database
	and to we go through all groups of size at most 2000
	excluding groups of size divisible by 128,
	and to compute for each group
	\begin{enumerate}
		\item whether it is solvable---%
			using \texttt{IsSolvable} function,
		\item whether it is a non-trivial direct product---%
			using a custom function \texttt{IsDirectProduct}
			(see the ancillary file \texttt{functions.m}),
			which searches for two normal subgroups with trivial intersection
			and whose sizes multiply to the size of the investigated group, 
		\item its order type---using a simple custom function \texttt{OrderType}.
	\end{enumerate}

\textbf{Step 2.}
	We parsed the outputs into a \texttt{.csv} table
	and handled the remaining tasks using various \texttt{Python} libraries.
	Given our objective of constructing an example
	through the direct product of groups,
	we focused solely on those groups that were not direct products themselves.
	
\textbf{Step 3.}
	We rephrased the question as follows.
	\begin{enumerate}
		\item First we converted the order types $o_G$ to exponent types $e_G$
			as described in \cref{sec:proof}.
		\item Then we used multiplicative M\"{o}bius inversion on $e_G$
			defining the \emph{revolved exponent type}
			$r_G: \mathbb{N} \to \mathbb{Q}$ as
			\begin{equation*}
				r_G(n) := \prod_{d|n} e_G(n/d)^{\mu(d)}.
			\end{equation*}
			The advantage of doing so was that
			$r_G(n) = 1$
			if $n\nmid E_G$.
			Indeed, for $k\mid E_G$ and $\gcd(m, E_G)=1$
			we get the following.
			\begin{align*}
				r_G(mk) & = \prod_{d|mk} e_G(mk/d)^{\mu(d)}
				= \prod_{e|m}\prod_{f|k} e_G(m/e \cdot k/f)^{\mu(ef)} \\
				& = \prod_{e|m}\prod_{f|k} e_G(k/f)^{\mu(e)\mu(f)}
				= \prod_{e|m}\Big(\prod_{f|k} e_G(k/f)^{\mu(f)}\Big)^{\mu(e)} \\
				& = r_G(k)^{\sum_{e|m} \mu(e)} =
				\begin{cases}
					r_G(k) &\text{ if } m = 1, \\
					1 &\text{ otherwise}.
				\end{cases}
			\end{align*}
			Additionally, we maintained the multiplicativity
			with respect to direct products:
			$r_{G_1\times G_2} = r_{G_1}\cdot r_{G_2}$.
			In order to operate with rational numbers
			we used a Python datatype \texttt{Fraction}.
		\item Finally, we factorised the revolved exponent types into their prime factors by defining
			$v_G(n, p)$ to be the exponent of a prime $p$
			in the factorisation of $r_G(n)$ into prime factors.
			Thus, for $G = \prod G_i^{m_i}$ we have
			\begin{equation*}
				v_G(n, p) = \sum m_i \cdot v_{G_i}(n, p).
			\end{equation*}
			Additionally, for a fixed $G$,
			the values of $v_G(n, p)$ are zero for all but finitely many combinations $(n, p)$. 
	\end{enumerate}
	
	At this point the question became:
	is there a positive rational linear combination
	$\sum k_i\cdot v_{N_i}$
	for non-solvable groups $N_i$
	in the rational span of $v_{S_i}$
	for $S_i$ solvable groups?

\textbf{Step 4.}
	An even simpler question,
	although not equivalent, arises:
	is any $v_N$ for a non-solvable group $N$
	in the span of $v_{S_i}$ for solvable groups $S_i$?
	This question is equivalent to
	determining whether the equation $Vx=v_N$ has a solution,
	where $V$ is a matrix of $v_{S_i}(n, p)$
	with rows indexed by pairs $(n, p)$
	and columns indexed by solvable groups $S_i$.
	A significant challenge lies in that the matrix $V$
	were in our case of size $9945 \times 100972$,
	albeit very sparse.

\textbf{Step 5.} We first attempted solving this numerically
	using Least Squares methods
	of the \texttt{SciPy} library \cite{2020SciPy-NMeth},
	which gave two conclusions.
	\begin{enumerate}
		\item For many of the non-solvable groups considered (e.g. $N=A_5$)
		the (numerical) projection of their $v_N$
		onto the column space of $V$
		was at distance about $1$ from $v_N$,
		meaning that $v_N$ was almost certainly not in the span.
		\item For some of the groups (e.g. the $\GL(3,2)$) the algorithm
		converged to a solution up to a small error of $10^{-6}$.
		This lead us to search for an exact solution to $Vx=v_N$
		for these groups.
	\end{enumerate}

\textbf{Step 6.} For $N=\GL(3, 2)$ we solved $Vx=v_N$ exactly
	using the symbolic computation library
	\texttt{SymPy} \cite{10.7717/peerj-cs.103}.
	A family of solutions was found; we chose one of them.

\textbf{Step 7.} We converted the solution back into the language of groups
	and checked that it holds using more direct methods.

\begin{landscape}
	\begin{table}
	\centering
	\small{
	\begin{tabular}{llll|l|llllllllllllllll}
		$i$ & $G_i$ & Id & $m_i$ &  $E_{G_i}$ &  1 &  2 &  3 &   4 &   6 &  7 &   8 &  12 &  14 &  21 &  24 &  28 &  42 &  56 &  84 &  168 \\
		\hline
		1 & $C_4$ &       (4, 1) &             9 &    4 &  1 &  2 &  1 &   4 &   2 &  1 &   4 &   4 &   2 &   1 &   4 &   4 &   2 &   4 &   4 &    4 \\
		2 & $D_3$ &       (6, 1) &             6 &    6 &  1 &  4 &  3 &   4 &   6 &  1 &   4 &   6 &   4 &   3 &   6 &   4 &   6 &   4 &   6 &    6 \\
		3 & $C_7$ &       (7, 1) &             1 &    7 &  1 &  1 &  1 &   1 &   1 &  7 &   1 &   1 &   7 &   7 &   1 &   7 &   7 &   7 &   7 &    7 \\
		4 & $D_4$ &       (8, 3) &             9 &    4 &  1 &  6 &  1 &   8 &   6 &  1 &   8 &   8 &   6 &   1 &   8 &   8 &   6 &   8 &   8 &    8 \\
		5 & $D_7$ &      (14, 1) &            18 &   14 &  1 &  8 &  1 &   8 &   8 &  7 &   8 &   8 &  14 &   7 &   8 &  14 &  14 &  14 &  14 &   14 \\
		6 & $\SL(2, 3)$ &      (24, 3) &            21 &   12 &  1 &  2 &  9 &   8 &  18 &  1 &   8 &  24 &   2 &   9 &  24 &   8 &  18 &   8 &  24 &   24 \\
		7 & $C_{24}\rtimes C_2$ &      (48, 6) &             3 &   24 &  1 & 14 &  3 &  28 &  18 &  1 &  32 &  36 &  14 &   3 &  48 &  28 &  18 &  32 &  36 &   48 \\
		8 & $C_7\rtimes D_4$ &      (56, 7) &             3 &   28 &  1 & 18 &  1 &  32 &  18 &  7 &  32 &  32 &  42 &   7 &  32 &  56 &  42 &  56 &  56 &   56 \\
		9 & $C_7\rtimes C_{12}$ &      (84, 1) &             6 &   84 &  1 &  2 & 15 &  16 &  30 &  7 &  16 &  72 &  14 &  21 &  72 &  28 &  42 &  28 &  84 &   84 \\
		10 & $\Dic_{21}$ &      (84, 5) &             6 &   84 &  1 &  2 &  3 &  44 &   6 &  7 &  44 &  48 &  14 &  21 &  48 &  56 &  42 &  56 &  84 &   84 \\
		11 & $C_7 \rtimes A_4$ &     (84, 11) &            21 &   42 &  1 &  4 & 57 &   4 &  60 &  7 &   4 &  60 &  28 &  63 &  60 &  28 &  84 &  28 &  84 &   84 \\
		12 & $C_7 \rtimes D_7$ &      (98, 4) &             2 &   14 &  1 & 50 &  1 &  50 &  50 & 49 &  50 &  50 &  98 &  49 &  50 &  98 &  98 &  98 &  98 &   98 \\
		13 & $C_4 \rtimes F_7$ &     (168, 9) &            21 &   84 &  1 & 30 & 15 &  32 & 114 &  7 &  32 & 144 &  42 &  21 & 144 &  56 & 126 &  56 & 168 &  168 \\
		14 & $C_{21}\rtimes D_4$ &    (168, 15) &             9 &   84 &  1 & 22 &  3 &  64 &  54 &  7 &  64 &  96 &  70 &  21 &  96 & 112 & 126 & 112 & 168 &  168 \\
		15 & $C_7\rtimes D_{12}$ &    (168, 17) &             6 &   84 &  1 & 50 &  3 &  64 &  54 &  7 &  64 &  96 &  98 &  21 &  96 & 112 & 126 & 112 & 168 &  168 \\
		16 & $F_8\rtimes C_3$ &    (168, 43) &             3 &   42 &  1 &  8 & 57 &   8 & 120 & 49 &   8 & 120 &  56 & 105 & 120 &  56 & 168 &  56 & 168 &  168 \\
		17 & $D_8\rtimes D_7$ &   (224, 106) &             3 &   56 &  1 & 52 &  1 &  96 &  52 &  7 & 128 &  96 & 112 &   7 & 128 & 168 & 112 & 224 & 168 &  224 \\
		18 & $C_7\rtimes D_{24}$ &    (336, 31) &             3 &  168 &  1 & 98 &  3 & 100 & 102 &  7 & 128 & 108 & 182 &  21 & 192 & 196 & 210 & 224 & 252 &  336 \\
	\end{tabular}
	\caption{The groups $G_i$, their multiplicities $m_i$,
		their exponents $E_{G_i}$,
		and their exponent types.}
	\label{tab:expontent_Gis}
	}
	\end{table}
	\begin{table}
		\vspace{3mm}
		\centering
		\tiny{
		\begin{subtable}{0.2\linewidth}
			\centering
			\begin{tabular}{l|l}
				1 & 1 \\
				2 & $2^{221}\cdot 3^{36}\cdot 5^{37}\cdot 7^9\cdot 11^9\cdot 13^3$ \\
				3 & $3^{126}\cdot 5^{27}\cdot 19^{24}$ \\
				4 & $2^{500}\cdot 3^{3}\cdot 5^{10}\cdot 7^3\cdot 11^6$ \\
			\end{tabular}
		\end{subtable}
		\hfill
		\begin{subtable}{0.2\linewidth}
			\centering
			\begin{tabular}{l|l}
				6 & $2^{215}\cdot 3^{174}\cdot 5^{34}\cdot 13^3\cdot 17^3\cdot 19^{21}$ \\
				7 & $7^{107}$ \\
				8 & $2^{530}\cdot 5^{4}\cdot 11^6$ \\
				12 & $2^{464}\cdot 3^{144}\cdot 5^{28}$ \\
			\end{tabular}
		\end{subtable}
		\hfill
		\begin{subtable}{0.2\linewidth}
			\centering
			\begin{tabular}{l|l}
				14 & $2^{191}\cdot 3^{33}\cdot 5^{9}\cdot 7^{113}\cdot 13^3$ \\
				21 & $3^{147}\cdot 5^3\cdot 7^{104}$ \\
				24 & $2^{488}\cdot 3^{132}\cdot 5^{28}$ \\
				28 & $2^{374}\cdot 3^{3}\cdot 7^{110}$ \\
			\end{tabular}
		\end{subtable}
		\hfill
		\begin{subtable}{0.2\linewidth}
			\centering
			\begin{tabular}{l|l}
				42 & $2^{185}\cdot 3^{177}\cdot 5^{3}\cdot 7^{104}$ \\
				56 & $2^{398}\cdot 7^{104}$ \\
				84 & $2^{347}\cdot 3^{114}\cdot 7^{104}$ \\
				168 & $2^{365}\cdot 3^{105}\cdot 7^{104}$ \\
			\end{tabular}
		\end{subtable}
		\caption{The exponent type of $G$.}
		\label{tab:Gexps}
		}
	\end{table}

	\begin{table}
		\centering
		\small{
		\begin{tabular}{llll|l|llllllllllllllll}
			$i$ & $H_i$ & Id &  $n_i$ &  $E_{H_i}$ &  1 &  2 &  3 &   4 &   6 &  7 &   8 &  12 &  14 &  21 &  24 &  28 &  42 &  56 &  84 &  168 \\
			\hline
			1 & $C_2$   &       (2, 1) &            21 &    2 &  1 &  2 &  1 &   2 &   2 &  1 &   2 &   2 &   2 &   1 &   2 &   2 &   2 &   2 &   2 &    2 \\
			2 & $C_3$   &       (3, 1) &             3 &    3 &  1 &  1 &  3 &   1 &   3 &  1 &   1 &   3 &   1 &   3 &   3 &   1 &   3 &   1 &   3 &    3 \\
			3 & $\Dic_3$    &      (12, 1) &             6 &   12 &  1 &  2 &  3 &   8 &   6 &  1 &   8 &  12 &   2 &   3 &  12 &   8 &   6 &   8 &  12 &   12 \\
			4 & $A_4$   &      (12, 3) &            21 &    6 &  1 &  4 &  9 &   4 &  12 &  1 &   4 &  12 &   4 &   9 &  12 &   4 &  12 &   4 &  12 &   12 \\
			5 & $\SD_{16}$  &      (16, 8) &             3 &    8 &  1 &  6 &  1 &  12 &   6 &  1 &  16 &  12 &   6 &   1 &  16 &  12 &   6 &  16 &  12 &   16 \\
			6 & $C_7\rtimes C_3$  &      (21, 1) &             4 &   21 &  1 &  1 & 15 &   1 &  15 &  7 &   1 &  15 &   7 &  21 &  15 &   7 &  21 &   7 &  21 &   21 \\
			7 & $D_{12}$   &      (24, 6) &             6 &   12 &  1 & 14 &  3 &  16 &  18 &  1 &  16 &  24 &  14 &   3 &  24 &  16 &  18 &  16 &  24 &   24 \\
			8 & $C_3\rtimes D_4$   &      (24, 8) &             6 &   12 &  1 & 10 &  3 &  16 &  18 &  1 &  16 &  24 &  10 &   3 &  24 &  16 &  18 &  16 &  24 &   24 \\
			9 & $\Dic_7$   &      (28, 1) &            15 &   28 &  1 &  2 &  1 &  16 &   2 &  7 &  16 &  16 &  14 &   7 &  16 &  28 &  14 &  28 &  28 &   28 \\
			10 & $F_7$   &      (42, 1) &            18 &   42 &  1 &  8 & 15 &   8 &  36 &  7 &   8 &  36 &  14 &  21 &  36 &  14 &  42 &  14 &  42 &   42 \\
			11 & $D_{21}$    &      (42, 5) &             6 &   42 &  1 & 22 &  3 &  22 &  24 &  7 &  22 &  24 &  28 &  21 &  24 &  28 &  42 &  28 &  42 &   42 \\
			12 & $D_{24}$    &      (48, 7) &             3 &   24 &  1 & 26 &  3 &  28 &  30 &  1 &  32 &  36 &  26 &   3 &  48 &  28 &  30 &  32 &  36 &   48 \\
			13 & $D_{28}$    &      (56, 5) &            27 &   28 &  1 & 30 &  1 &  32 &  30 &  7 &  32 &  32 &  42 &   7 &  32 &  56 &  42 &  56 &  56 &   56 \\
			14 & $\Dic_7\rtimes C_6$   &    (168, 11) &             3 &   84 &  1 & 18 & 15 &  32 & 102 &  7 &  32 & 144 &  42 &  21 & 144 &  56 & 126 &  56 & 168 &  168 \\
			15 & $C_{14}.A_4$   &    (168, 23) &            21 &   84 &  1 &  2 & 57 &   8 & 114 &  7 &   8 & 120 &  14 &  63 & 120 &  56 & 126 &  56 & 168 &  168 \\
			16 & $\GL(3,2)$   &    (168, 42) &             3 &   84 &  1 & 22 & 57 &  64 &  78 & 49 &  64 & 120 &  70 & 105 & 120 & 112 & 126 & 112 & 168 &  168 \\
			17 & $C_7\rtimes F_7$   &    (294, 10) &             2 &   42 &  1 & 50 & 15 &  50 & 162 & 49 &  50 & 162 &  98 & 147 & 162 &  98 & 294 &  98 & 294 &  294 \\
			18 & $D_{12}.D_7$   &    (336, 36) &             3 &  168 &  1 & 14 &  3 & 100 &  18 &  7 & 128 & 108 &  98 &  21 & 192 & 196 & 126 & 224 & 252 &  336 \\
		\end{tabular}
		\caption{The groups $H_i$, their multiplicities $n_i$,
			their exponents $E_{H_i}$,
			and their exponent types.}
		\label{tab:expontent_His}
		}
	\end{table}
	\begin{table}
		\vspace{3mm}
		\centering
		\tiny{
			\begin{subtable}{0.2\linewidth}
				\centering
				\begin{tabular}{l|l}
					1 & 1 \\
					2 & $2^{221}\cdot 3^{36}\cdot 5^{37}\cdot 7^{9}\cdot 11^{9}\cdot 13^{3}$ \\
					3 & $3^{126}\cdot 5^{27}\cdot 19^{24}$ \\
					4 & $2^{500}\cdot 3^{3}\cdot 5^{10}\cdot 7^{3}\cdot 11^{6}$ \\
				\end{tabular}
			\end{subtable}
			\hfill
			\begin{subtable}{0.2\linewidth}
				\centering
				\begin{tabular}{l|l}
					6 & $2^{215}\cdot 3^{174}\cdot 5^{34}\cdot 13^{3}\cdot 17^{3}\cdot 19^{21}$ \\
					7 & $7^{107}$ \\
					8 & $2^{530}\cdot 5^{4}\cdot 11^{6}$ \\
					12 & $2^{464}\cdot 3^{144}\cdot 5^{28}$ \\
				\end{tabular}
			\end{subtable}
			\hfill
			\begin{subtable}{0.2\linewidth}
				\centering
				\begin{tabular}{l|l}
					14 & $2^{191}\cdot 3^{33}\cdot 5^{9}\cdot 7^{113}\cdot 13^{3}$ \\
					21 & $3^{147}\cdot 5^{3}\cdot 7^{104}$ \\
					24 & $2^{488}\cdot 3^{132}\cdot 5^{28}$ \\
					28 & $2^{374}\cdot 3^{3}\cdot 7^{110}$ \\
				\end{tabular}
			\end{subtable}
			\hfill
			\begin{subtable}{0.2\linewidth}
				\centering
				\begin{tabular}{l|l}
					42 & $2^{185}\cdot 3^{177}\cdot 5^{3}\cdot 7^{104}$ \\
					56 & $2^{398}\cdot 7^{104}$ \\
					84 & $2^{347}\cdot 3^{114}\cdot 7^{104}$ \\
					168 & $2^{365}\cdot 3^{105}\cdot 7^{104}$ \\
				\end{tabular}
			\end{subtable}
			\caption{The exponent type of $H$,
				equal to the exponent type of $G$.}
			\label{tab:Hexps}
		}
	\end{table}
\end{landscape}

%%%%%%%%%%%%%%%%%%%%%% REFERENCES HERE

\bibliographystyle{plain}
\bibliography{Bibliography}

\begin{thebibliography}{10}

\bibitem{besche2022smallgrp}
Hans~Ulrich Besche, Bettina Eick, and Eamonn O’Brien.
\newblock {Small Groups Library}.
\newblock \url{https://magma.maths.usyd.edu.au/magma/handbook/text/779}.
\newblock Accessed: 2024-03-14.

\bibitem{MR1484478}
Wieb Bosma, John Cannon, and Catherine Playoust.
\newblock The {M}agma algebra system. {I}. {T}he user language.
\newblock {\em J. Symbolic Comput.}, 24(3-4):235--265, 1997.
\newblock Computational algebra and number theory (London, 1993).

\bibitem{guralnick1993same}
Robert~M Guralnick and Al~Weiss.
\newblock Same genus and embeddings of groups.
\newblock In {\em Linear Algebraic Groups and Their Representations: Conference
  on Linear Algebraic Groups and Their Representations, March 25-28, 1992, Los
  Angeles, California}, volume 153, page~21. American Mathematical Soc., 1993.

\bibitem{herzog2022another}
Marcel Herzog, Patrizia Longobardi, and Mercede Maj.
\newblock Another criterion for solvability of finite groups.
\newblock {\em Journal of Algebra}, 597:1--23, 2022.

\bibitem{khukhro2014unsolved}
EI~Khukhro and VD~Mazurov.
\newblock Unsolved problems in group theory. {The Kourovka} notebook.
\newblock {\em arXiv preprint arXiv:1401.0300}, 2014.

\bibitem{lazorec2023average}
Mihai-Silviu Lazorec and Marius T{\u{a}}rn{\u{a}}uceanu.
\newblock On the average order of a finite group.
\newblock {\em Journal of Pure and Applied Algebra}, 227(4):107276, 2023.

\bibitem{10.7717/peerj-cs.103}
Aaron Meurer, Christopher~P. Smith, Mateusz Paprocki, Ond\v{r}ej
  \v{C}ert\'{i}k, Sergey~B. Kirpichev, Matthew Rocklin, AMiT Kumar, Sergiu
  Ivanov, Jason~K. Moore, Sartaj Singh, Thilina Rathnayake, Sean Vig, Brian~E.
  Granger, Richard~P. Muller, Francesco Bonazzi, Harsh Gupta, Shivam Vats,
  Fredrik Johansson, Fabian Pedregosa, Matthew~J. Curry, Andy~R. Terrel,
  \v{S}t\v{e}p\'{a}n Rou\v{c}ka, Ashutosh Saboo, Isuru Fernando, Sumith Kulal,
  Robert Cimrman, and Anthony Scopatz.
\newblock Sympy: symbolic computing in python.
\newblock {\em PeerJ Computer Science}, 3:e103, January 2017.

\bibitem{shen2012groups}
Rulin Shen.
\newblock On groups with given same-order types.
\newblock {\em Communications in Algebra}, 40(6):2140--2150, 2012.

\bibitem{shen2023thompson}
Rulin Shen, Wujie Shi, and Feng Tang.
\newblock On {Thompson} problem.
\newblock {\em arXiv preprint arXiv:2308.07183}, 2023.

\bibitem{shi2024quantitative}
Wujie Shi.
\newblock Quantitative characterization of finite simple groups: a complement.
\newblock {\em arXiv preprint arXiv:2309.06362}, 2024.

\bibitem{vasil2009characterization}
Andrey~V Vasil’ev, Mariya~Aleksandrovna Grechkoseeva, and Victor~Danilovich
  Mazurov.
\newblock Characterization of the finite simple groups by spectrum and order.
\newblock {\em Algebra and Logic}, 48(6):385--409, 2009.

\bibitem{2020SciPy-NMeth}
Pauli Virtanen, Ralf Gommers, Travis~E. Oliphant, Matt Haberland, Tyler Reddy,
  David Cournapeau, Evgeni Burovski, Pearu Peterson, Warren Weckesser, Jonathan
  Bright, St{\'e}fan~J. {van der Walt}, Matthew Brett, Joshua Wilson, K.~Jarrod
  Millman, Nikolay Mayorov, Andrew R.~J. Nelson, Eric Jones, Robert Kern, Eric
  Larson, C~J Carey, {\.I}lhan Polat, Yu~Feng, Eric~W. Moore, Jake
  {VanderPlas}, Denis Laxalde, Josef Perktold, Robert Cimrman, Ian Henriksen,
  E.~A. Quintero, Charles~R. Harris, Anne~M. Archibald, Ant{\^o}nio~H. Ribeiro,
  Fabian Pedregosa, Paul {van Mulbregt}, and {SciPy 1.0 Contributors}.
\newblock {{SciPy} 1.0: Fundamental Algorithms for Scientific Computing in
  Python}.
\newblock {\em Nature Methods}, 17:261--272, 2020.

\end{thebibliography}
	
\end{document}